\documentclass[12pt]{amsart}

\usepackage{amsmath,amssymb,amsthm}
\usepackage{graphicx}
\usepackage{enumerate}
\usepackage{color}
\usepackage[german, english]{babel}
\usepackage[all]{xy}

\newtheorem{thm}{Theorem}[]
\newtheorem*{thm*}{Theorem}
\newtheorem{lem}[thm]{Lemma}

\newcommand{\param}{{\mathchoice{\mkern1mu\mbox{\raise2.2pt\hbox{$
\centerdot$}}
\mkern1mu}{\mkern1mu\mbox{\raise2.2pt\hbox{$\centerdot$}}\mkern1mu}{
\mkern1.5mu\centerdot\mkern1.5mu}{\mkern1.5mu\centerdot\mkern1.5mu}}}

\renewcommand{\setminus}{{\smallsetminus}}

\renewcommand \color [2][]{}

\begin{document}

\title       {The Chillingworth Class is a Signed Stable Length.}
\author   {Ingrid Irmer}
\address {Department of Mathematics\\
               National University of Singapore\\
               10 Lower Kent Ridge Road
               Singapore 119076}
\email      {matiim@nus.edu.sg}
\maketitle

\begin{abstract} An orientation is defined on a family of curve graphs on which the Torelli group acts. It is shown that the resulting signed stable length of an element of the Torelli group is a cohomology class. This cohomology class is half the dual of the contraction of the Johnson homomorphism, the so-called ``Chillingworth class''.
\end{abstract}

\section{Introduction}
Let $S_{g}$ be a closed, oriented surface with genus $g\geq 3$, and $S_{g,1}$ an oriented surface with genus $g\geq 3$ and one boundary curve. The \textit{mapping class group} of $S_{g}$, denoted Mod($S_{g}$) is the group of isotopy classes of orientation preserving homeomorphisms of $S_{g}$. The group Mod($S_{g,1}$) is defined similarly, with the added condition that the homeomorphisms act as the identity on the boundary of $S_{g,1}$. The group Mod($S_{g}$) induces an action on $H_{1}(S_{g},\mathbb{Z})$, and the \textit{Torelli group}, $\mathcal{T}_{g}$ is the kernel of this action. Similarly for $\mathcal{T}_{g,1}$.\\

\textbf{Johnson Homomorphisms.} The \textit{Johnson homomorphisms}, \cite{Johnson}, are homomorphisms 
\begin{equation}
t_{g,1}:\mathcal{T}_{g,1}\rightarrow \wedge^{3}H_{1}(S_{g,1};\mathbb{Z}) \text{ and }t_{g}:\mathcal{T}_{g}\rightarrow \wedge^{3}H_{1}(S_{g};\mathbb{Z})/H_{1}(S_{g};\mathbb{Z})
\end{equation}
The Johnson homomorphisms are amongst the most basic structures in the study of the Torelli group, for example, for understanding group homology, \cite{Johnsonsurvey} and 3-manifold theory, \cite{Morita}. In \cite{Chill2}, Chillingworth used a notion of winding number to define the \textit{Chillingworth class}; a map from the Torelli group to an integral first cohomology class of $S_{g,1}$ or $S_{g}$. It was shown that the vanishing of this map is a necessary condition for a mapping class in the Torelli group to be in the kernel of the Johnson homomorphism, i.e. the \textit{Johnson kernel}.  Johnson showed in \cite{Johnson} that the Chillingworth class is dual to a contraction of the Johnson homomorphism. Some surveys on the Johnson homomorphism can be found in \cite{Johnsonsurvey} and \cite{Mappingclassbible}.\\

\textbf{Curve Graphs.} Let $C(S_{g,1}; h)$ be a graph with vertices consisting of oriented isotopy classes of simple closed curves in the primitive homology class $h\in H_{1}(S_{g,1}; \mathbb{Z})$. Any two vertices representing disjoint curves whose difference bounds one or more subsurfaces of $S_{g,1}$ of genus 1 are connected by an edge. The graph $C(S_{g},h)$ is defined analogously.\\

\begin{figure}
\begin{center}
\def\svgwidth{13cm}
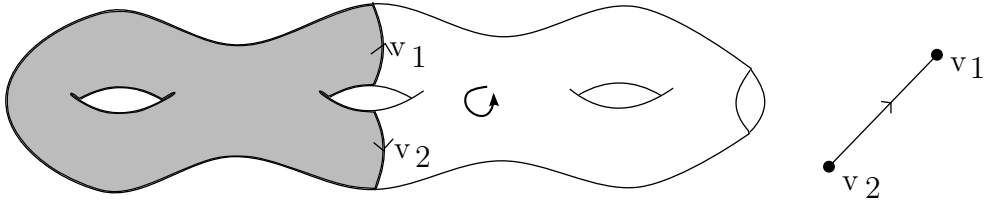
\label{curvegraph}
\caption{The shaded subsurface represents an edge of the curve graph $C(S_{3,1},h)$}
\end{center}
\end{figure}

Since there are no bounding pairs in $S_2$, $C(S_{2}; h)$ does not have any edges. For this reason, and also because the Johnson homomorphism, and hence the Chillingworth class, is zero in genus 2, \cite{Johnson}, we will not be considering surfaces with genus two.\\


\textbf{Signed distance.} There is of course the usual combinatorial distance $d(v_{1},v_{2})$ defined on the graphs $C(S_{g,1}; h)$ and $C(S_{g}; h)$. Since a path in $C(S_{g,1}; h)$ or $C(S_{g}; h)$ defines a surface in a 3-manifold, as described in \cite{Me}, the combinatorial distance is a special case of the Thurston norm, defined in \cite{Thurstonnorm}. To  turn this norm into a cohomology class, orientation information, in the form of a signed distance, $d_{s}(v_{1},v_{2})$, on $C(S_{g,1}; h)$ or $C(S_{g}; h)$, is utilised.\\

Calculating Thurston norm, or the very closely related Gromov norm, involves finding an infimum analogous to the infimum that is a stable length of a group action on a metric space such as a graph. Let $\eta$ be an endomorphism of a free group $F$. In \cite{Cal}, Section 9, Gromov norm of the HNN extension of $F$, ${F_{*}}_{\eta}$, was shown to be a ``translation length'' of $\eta$ acting on $K(F,1)$. The ``translation length'' in this context was defined to be a \textit{stable} commutator length, whereas translation length, (also called stable length) on $C(S_{g,1},h)$ is a commutator length. The action of an element of the mapping class group of $S_{g,1}$ on the fundamental group of the surface gives an endomorphism of a free group. In \cite{Cal} it was emphasized that the Chillingworth class is a rotation quasi-morphism. Rotation quasi-morphisms are defined in \cite{scl} and are used to estimate, or sometimes calculate, stable commutator length.\\

Recall the convention that the boundary of a subsurface of $S_{g,1}$ or $S_g$ is oriented in such a way that it has the subsurface to its left. A subsurface of $S_{g,1}$ or $S_g$ is assigned +1 if it is oriented as a subsurface of $S_{g,1}$ or $S_g$ and -1 otherwise. In particular, this gives an orientation to the edges of the curve graphs, as shown in figure \ref{curvegraph}.\\

Choose a path $\gamma$ connecting $v_1$ and $v_2$ in $C(S_{g,1}; h)$ or $C(S_{g};h)$, and let $l_{s}(\gamma)$ be the number of edges traversed by $\gamma$ in the positive direction minus the number of edges traversed in the negative direction. In Lemma \ref{sd} it is shown that, in $C(S_{g,1},h)$, $l_{s}(\gamma)$ is independent of the choice of $\gamma$. In this case, the \textit{signed distance}, $d_{s}(v_{1},v_{2})$, is defined to be $l_{s}(\gamma)$. For $S_g$, the signed distance is only independent mod $g-1$ of the choice of $\gamma$, and is defined to be this equivalence class mod $g-1$. The closed surface $S_3$ is a special case, because as shown in figure \ref{boundingpair}, a bounding pair cuts $S_3$ into two subsurfaces of genus one, so it is not clear how to orient edges of the curve graph. However, since $+1$ is congruent to $-1$ mod 3-1, any arbitrary choice of edge orientation will determine the same signed length of paths mod 2.\\

\begin{figure}
\centering
\includegraphics[width=13cm]{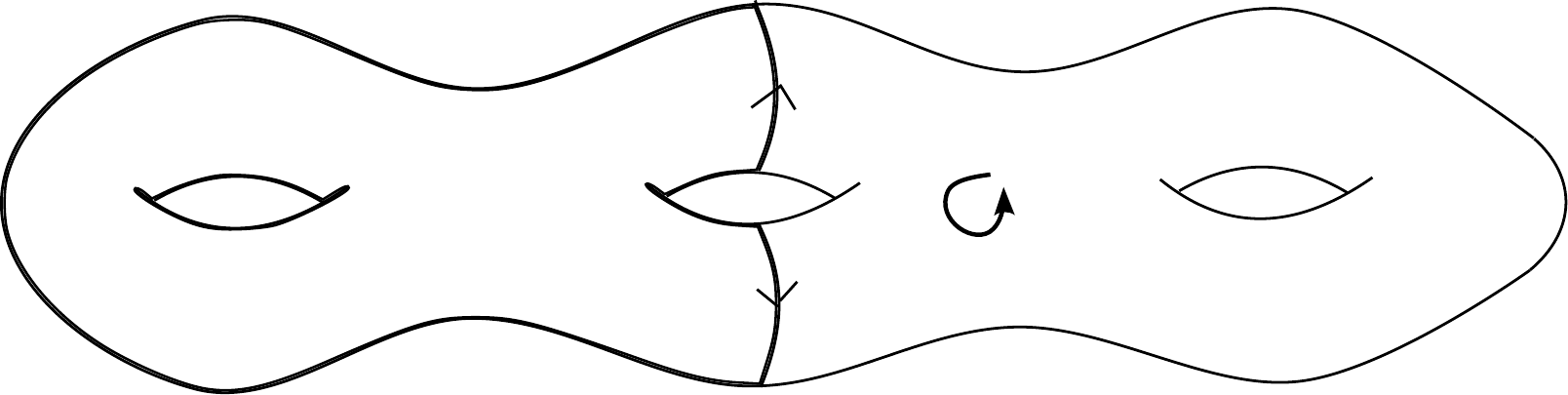}
\caption{A bounding pair cuts $S_3$ into two surfaces of genus 1.}
\label{boundingpair}
\end{figure}

Since the respective Torelli groups preserve orientation and act on $C(S_{g},h)$ and $C(S_{g,1},h)$ by isometry, signed distance is also preserved by the Torelli group.\\

\textbf{Signed Stable Length.} Choose a primitive homology class $h$, and define the signed stable length $\phi(\tau,h)$ of an element $\tau$ of $\mathcal{T}_{g,1}$ acting on $C(S_{g,1}; h)$ as follows:
\begin{equation}
\phi(\tau,h)=lim_{n\rightarrow \infty}\frac{d_{s}(v,\tau^{n}v)}{n}
\end{equation}
where $v$ is any vertex of $C(S_{g,1}; h)$. It is not a priori clear that this is well defined, however this will be shown to follow from lemma 2.\\

By a construction due to Hatcher, \cite{Cycle}, $C(S_{g,1}; h)$ and $C(S_{g}; h)$ are seen to be connected. Since the respective Torelli groups act on the curve graphs $C(S_{g,1}; h)$ and $C(S_{g}; h)$ by isometry, it follows immediately that stable length, signed or otherwise, is locally independent of the choice of vertex $v$, and hence globally independent by connectivity.\\

Lemma \ref{nolimit} states that $\phi(\tau,h)=d_{s}(v,\tau v)$, for any vertex $v$ in $C(S_{g,1}; h)$. For closed surfaces, we have to make sense of what it means to stabilise a quantity that is only defined mod $g-1$. In analogy with open surfaces, for $\tau\in \mathcal{T}_{g}$ and $h$ a primitive element of $H_{1}(S_{g};\mathbb{Z})$, $\phi(\tau,h):=d_{s}(v,\tau v)$. For both $C(S_{g,1}; h)$ and $C(S_{g}; h)$, if $h$ is not primitive, i.e. $h=\lambda [c]$ for a primitive curve $c$, then $\phi(\tau,h):=\lambda \phi(\tau,[c])$.\\

Observe that, unlike for the mapping class group acting on the curve complex, a different stable length is obtained depending on the choice of homology class defining the graph on which the element of the Torelli group acts. An illustrative example is given by the action of a bounding pair map $T_{a}T^{-1}_{b}$. Here $T_{a}$ denotes a Dehn twist around the curve $a$, and it is assumed that the oriented multicurve $a-b$ is the boundary of a subsurface of genus 1. A bounding pair map is not pseudo-Anosov, so its action on the curve complex has stable length zero. The calculation in Section 3.4 gives a formula for signed stable length of $T_{a}T^{-1}_{b}$ on $C(S_{g,1},h)$, depending on the algebraic intersection number of $[a]$ with $h$. For bounding pair maps, arguing as in Subsection \ref{kernel}, it is not hard to show that signed stable length is equal to unsigned stable length on $C(S_{g,1},h)$.\\


It is finally possible to state the theorem of this paper.

\begin{thm}
The signed stable length $\phi$ is half the Chillingworth class.
\label{Maintheorem}
\end{thm}


\section{Acknowledgements} Thanks to Andy Putman and Dan Margalit for discussions on earlier incarnations of this result, and to Danny Calegari, Allen Hatcher and a very thorough reviewer for comments and improvements. This work was funded by a MOE AcRF-Tier 2 WBS grant Number R-146-000-143-112.

\section{Some Background on the Chillingworth Class}
In this section, suppose all curves are closed, simple, oriented and with continously varying, nowhere zero tangent vector. Let $X$ be a nowhere zero vector field on $S_{g,1}$. In \cite{Chill}, Chillingworth defined the winding number of a curve $\gamma$ with respect to $X$. When presented precisely and in all generality, the definition is quite long, therefore only the general idea of the special case needed for studying the mapping class group will be given here. \\

Intuitively, the winding number $\omega_{X}(\gamma)$ is the number of times the tangent vector to $\gamma$ rotates relative to $X$ as $\gamma$ is traversed once in its positive direction. To relate winding numbers to cohomology classes, it was shown that $\omega_{X_{1}}(\gamma)-\omega_{X_{2}}(\gamma)$ only depends on the integral homology class of $\gamma$.\\

In \cite{Chill2} applications of winding number to the study of mapping class groups were discussed. At the end of this paper the conjecture was made that for $\tau \in \mathcal{T}_{g,1}$, if $\omega_{\tau^{*}X}(\gamma)-\omega_{X}(\gamma)=0$ for any curve $\gamma$, then $\tau$ must be in the Johnson kernel. (Back then, the Johnson kernel was not yet known to be a kernel, but was considered interesting in its own right as the normal subgroup generated by Dehn twists around separating curves). In \cite{Johnson}, Johnson showed this conjecture to be false, and started calling the cohomology class $d(\tau^{*}X, X)$, defined by 
\begin{equation*}
\langle d(\tau*X,X),[\gamma] \rangle = \omega_{\tau^{*}X}(\gamma)-\omega_{X}(\gamma)
\end{equation*}
the ``Chillingworth class''. This cohomology class can be shown to be independent of the choice of $X$.\\

In Section 6 of \cite{Chill}, winding numbers on \textit{closed} surfaces were defined. For closed surfaces of genus greater than or equal to two, there are no nonvanishing vector fields, so suppose $X$ only has one zero; call this point $p$. When a smooth homotopy moves the curve $\gamma$ over $p$, by the Poincar\'e-Hopf index theorem, the winding number of $\gamma$ is changed by $\pm(2-2g)$. It follows that for closed surfaces, the Chillingworth class is only defined mod $2g-2$.\\

\section{Proof of Theorem}
\begin{proof} The proof of the theorem is broken up into four parts.

\subsection{Signed distance}\label{Sectionone}
In this subsection, some basic properties of signed distance on $C(S_{g,1}; h)$ and $C(S_{g}; h)$ are established.\\

A \textit{curve} is an oriented isotopy class of closed loops not homotopic into the boundary of $S_{g,1}$. When this does not lead to ambiguity, the same symbol will be used for a curve and a particular representative of the isotopy class. Also, in order to show that signed length only depends on the endpoints of the path, or only on the endpoints of the path mod $g-1$, properties of surfaces in the 3-manifolds $S_{g}\times I$ and $S_{g,1}\times I$ will be used. In both cases, the surface $S_{g}$ or $S_{g,1}$ maps into the 3-manifold. Where this does not cause confusion, the same notation will be used for a curve on a surface and the isotopy class of its image in the 3-manifold.\\

A surface invariant will now be defined that is important in the proof of the next lemma.\\

\textbf{The pre-image function.} This definition is taken from \cite{Me}. Let $H$ be a connected, immersed surface in $S_{g}\times I$. Curve graph distances are related to surfaces via the \textit{pre-image function} $p_{H}:S_{g}\times 0\setminus \pi(\partial H) \rightarrow \mathbb{Z}$. Suppose $P:=S_{g}\times I$ and $B$ are open sets in $(S_{g}\times \left\{0\right\})\setminus \pi(\partial H)$. Algebraic intersection number provides a map $H_{2}(P, \partial H)\times H_{1}(P,B)\to \mathbb{Z}$. For $x$ in $(S_{g}\times \left\{0\right\})\cap B$,
\begin{equation}
p_{H,B}(x):=\hat{i}(H, x\times I)
\end{equation}

For all $x\subset S_{g}\times 0 \setminus \pi(\partial H)$ there is a choice of $B$ such that $x\subset B$. To see that the function $p_{H,B}$ does not depend on the choice of open set $B$, note that if $B\subset B^{'}\subset (S_{g}\times 0)\setminus \pi(\partial H)$, it follows from the naturality of the intersection pairing with respect to inclusions (\cite{Dold} Proposition 1.3.4) that the diagram below commutes.\\
$$
\xymatrixrowsep{1in}
\xymatrixcolsep{1in}
\xymatrix{B \ar@{^{(}->}[r] \ar[dr]^{p_{H,B}}& B^{'}\ar[d]^{p_{H,B^{'}}}\\
            & \mathbb{Z}}
$$


If $B_{1}\subset B_{2} \subset B_{3}\ldots$ and $D_{1}\subset D_{2} \subset \ldots$ are any two sequences of open subsets of $(S_{g}\times \{0\})\setminus \pi(\partial H)$ with direct limit $(S_{g}\times \{0\})\setminus \pi(\partial H)$, it follows that they both give rise to the same function $p_{H}$ as the sequence $B_{1}\cap D_{1}\subset B_{2}\cap D_{2}\subset \ldots$, which also has direct limit $(S_{g}\times \{0\})\setminus \pi(\partial H)$. Hence $p_{H}$ is well defined on $(S_{g}\times \{0\})\setminus \pi(\partial H)$, and is extended to an upper semi-continuous function defined on all of $S_{g}\times \{0\}$ .\\

The pre-image function is defined analogously for $S_{g,1}$.\\



\textbf{Surfaces and curve graph paths. }Let $c_1$ and $c_2$ be curves in general and minimal position representing the vertices $v_1$ and $v_2$ respectively. Suppose $\pi$ is the projection of $S_{g}\times I$ onto $S_{g}\times \{0\}$, or $S_{g,1}\times I$ onto $S_{g,1}\times \{0\}$. In \cite{Me}, Section 2.1, it is explained how to construct a surface in $S_{g}\times I$ with boundary $c_{2}-c_{1}$ from an oriented path in a curve graph passing from vertex $v_{1}$ to vertex $v_{2}$. A simpler case of this is sketched below. Strictly speaking, the objects so obtained are often only cell complexes, but are homotopic to immersed surfaces with boundary. \\

Suppose a path in $C(S_{g}; h)$ has $j$ edges labelled $i=1,\ldots, j$, where $j>1$. Informally, each edge of $C(S_{g}; h)$ represents an oriented subsurface $F_i$, of $S_{g}$, which maps into $S_{g}\times \{\frac{i}{j+1}\}\subset S_{g}\times I$. Pairs of homotopic curves with opposite orientations on the boundaries of $F_{i}$ and $F_{i+1}$ are connected up by gluing in cylinders. Similarly, the boundary of the surface is put into the boundary of the 3-manifold by gluing on cylinders. The same construction works for $S_{g,1}\times I$.\\

A surface constructed from the path $\gamma$ will be called a \textit{trace surface} of $\gamma$. The trace surface of an \textit{oriented} path in $C(S_{g}; h)$ or $C(S_{g,1}; h)$ is unique up to homotopy.\\

\textbf{Remark}. The vertices $v_1$ and $v_2$ are connected by paths of finite length. The correct intuition is that each edge of a given path represents a surface with boundary that contributes at most $\pm 1$ to the preimage function on any given subsurface. It follows that the pre-image function is finite, and each level set has a finite number of connected components. A simple means of calculating the pre-image function and hence verifying finiteness is given in \cite{Me} Section 4.\\

\textbf{Euler Integrals.} Throughout this paper we will want to compute a ``signed genus'', coming from a ``signed Euler characteristic''. The \textit{signed genus} of the trace surface of $\gamma$ is equal to the signed length of $\gamma$. To calculate signed Euler characteristic in such a way that the signs work out automatically, a very convenient notation makes use of integration with respect to Euler characteristic. The formulation used here is taken from \cite{Ghrist}; proofs are given in \cite{Shapira}. The reference \cite{Rota} is a comprehensive introduction to invariant measures with many examples of novel applications.\\

Definition 2.2 of \cite{Ghrist}. A collection $\mathcal{A}$ of subsets of a topological space $X$ is said to be \textit{tame} if $\mathcal{A}$ is closed with respect to the operations of finite intersection, finite union and complement, and all elements of $\mathcal{A}$ possess well-defined Euler characteristics. \\

An indicator function $1_U$ is the function that is equal to 1 on the set $U$ and zero elsewhere. Let $R$ be a commutative coefficient ring; in this paper, we will take $R$ to be the integers.\\

Lemma 2.4 of \cite{Ghrist}. Let $\mathcal{A}$ be a tame collection and $f=\Sigma\lambda_{\alpha} 1_{U_{\alpha}}$ a finite $R$-linear combination of indicator functions of elements $U_{\alpha}\in \mathcal{A}$. Then the integral of $f$ is defined to be
\begin{equation*}
\int_{X}fd\chi:=\Sigma_{\alpha}\lambda_{\alpha}\chi(U_{\alpha})
\end{equation*}
The value of this integral is independent of the way $f$ is written as a finite sum of indicator functions.\\

\textbf{Remark}. Pre-image functions were defined to be upper semi-continuous because we will want to relate their Euler integrals to signed genus of trace surfaces. Although the boundary of a trace surface has Euler characteristic zero, its projection to $S_{g}\times \{0\}$ or $S_{g,1}\times \{0\}$ does not, so it is necessary to be careful what value to assign it. Upper semi-continuity ensures that a union of surfaces gives rise to a pre-image function that is the sum of the pre-image functions of the surfaces.

\begin{lem}
\label{sd}
Let $v_1$ and $v_2$ be two vertices of $C(S_{g,1}; h)$. All paths in $C(S_{g,1}; h)$ connecting $v_1$ to $v_2$ have the same signed length. In $C(S_{g}; h)$, all paths connecting $v_1$ to $v_2$ have the same signed length mod $g-1$.
\end{lem}
\begin{proof}
Lemma 10 of \cite{Me} shows that any surface in $S_{g}\times I$ with boundary $c_{2}-c_{1}$ has the same pre-image function up to an additive constant. The same argument applies to $S_{g,1}\times I$, although it is not explicitly stated. \\

We first prove the statement for open surfaces, and then explain why closed surfaces are different. Let $\gamma_{1}$  and $\gamma_{2}$ be paths connecting $v_1$ to $v_2$ in $C(S_{g,1}; h)$, where $H_1$ is the trace surface of $\gamma_{1}$ and $H_2$ the trace surface of $\gamma_{2}$. As before, $\partial H_{1}=\partial H_{2}=c_{2}-c_{1}$. By assumption, neither $c_1$ nor $c_2$ is homotopic to $\partial S_{g,1}$. From the definition of pre-image function, it follows that $p_{H}$ must be zero on some neighbourhood of $\partial S_{g,1}$. Since $p_{H_{1}}-p_{H_{2}}$ is a constant, it follows that $p_{H_{1}}=p_{H_{2}}$.\\ 

Recall that the surface $H_{1}$ is homotopic to a union of genus one subsurfaces $F_i$, each of which projects one to one onto an embedded, incompressible subsurface of $S_{g,1}\times {0}$ of genus 1. In order to define the pre-image function of $F_i$ for some $i$, strictly speaking the boundary of $F_i$ should be put in the boundary of the 3-manifold. This detail will be ignored, because it is not hard to see that a homotopy taking $F_{i}$ to $F^{'}_{i}$ can be found, such that $p_{F^{'}_{i}}$ is one on the interior of the subsurface $\pi(F_{i})$ of $S_g$ or $S_{g,1}$ and zero on $S_{g}\setminus \pi(F_{i})$ or $S_{g,1}\setminus \pi(F_{i})$. The pre-image function $p_H$ is the sum of the preimage functions of the $F_i$, except along the curves where annuli are glued in to join the $F_i$ together, over which $p_H$ varies continuously.\\

The signed genus of $F_i$ is obtained by
\begin{equation*}
\frac{-1}{2}\int_{S_{g}}p_{F_{i}}d\chi
\end{equation*}

Depending on the orientation of $F_i$, this will be plus or minus one. Since integrals preserve finite linear combinations,  it follows that the signed length of $\gamma_{1}$ is equal to 

\begin{equation}
\label{hom2}
\frac{-1}{2}\int_{S_{g,1}}p_{H_{1}}d\chi
\end{equation}

which is independent of the choice of $\gamma_1$ because $p_{H_{1}}$ is.\\

The proof is the same for closed surfaces, except that paths can be chosen such that the pre-image functions of the trace surfaces differ by any integer. To see this, we show how to increase or decrease the pre-image function by one. Suppose $v$ is a vertex on $\gamma_1$ corresponding to the curve $c$. Cut $H_1$ along the curve $c$, and glue in a copy of a surface homotopic in $S_{g}\times I$ to $(S_{g}\times \{1/2\})\setminus (c\times \{1/2\}$. The orientation of the surface glued in will determine whether the pre-image function increases or decreases by one.\\


To see that we need to take mod $g-1$ instead of mod $g$, note that, when $S_g$ is cut into two pieces by a pair of disjoint, homologous curves, the genus of these two pieces sums to $g-1$, not $g$, as illustrated in figure \ref{boundingpair}. \\

\end{proof}

\textbf{Remark}. The ambiguity in signed distance on $C(S_{g};h)$ will be related to the observation that Chillingworth's winding numbers are only defined mod $2g-2$ for $S_{g}$. \\

When defining signed stable length on $C(S_{g},h)$, it is necessary to make sense of the limit. If the limit exists at all, is it independent of the choice of path mod $g-1$? The next lemma will be used to resolve these problems.

\begin{lem}
 Let $v$ and $w$ be any two vertices of $C(S_{g},h)$, and $\tau$ an element of $\mathcal{T}_g$. Then for any $n\in \mathbb{N}$,
\begin{equation*}
d_{s}(w,\tau^{n}w)\equiv d_{s}(v,\tau^{n}v)\equiv nd_{s}(v,\tau v) \text{ mod } g-1
\end{equation*}
Similarly, when $w$ and $v$ are vertices of $C(S_{g,1};h)$ and $\tau\in \mathcal{T}_{g,1}$,
\begin{equation*}
d_{s}(w,\tau^{n}w)= d_{s}(v,\tau^{n}v)= nd_{s}(v,\tau v) 
\end{equation*}
\label{nolimit}
\end{lem}
\begin{proof}
From Lemma \ref{sd}, any path connecting $w$ to $\tau^{n}w$ can be used to calculate $d_{s}(w,\tau^{n}w)$. So suppose $\gamma$ is a path connecting $w$ to $\tau^{n}w$ passing through $v$ and $\tau^{n}v$, in that order. Suppose also that the unoriented subpath of $\gamma$ connecting $\tau^{n}v$ to $\tau^{n}w$ is the image under $\tau^{n}$ of the subpath connecting $w$ to $v$. Reversing the orientation of a path changes the sign of its signed length, and the action of $\tau$ preserves signed distance, so the signed length of the subpath connecting $w$ to $v$ cancels out the signed length of the subpath connecting $\tau^{n}v$ to $\tau^{n}w$. It follows that $d_{s}(w,\tau^{n}w)\equiv d_{s}(v,\tau^{n}v)$.\\

To calculate $d_{s}(v,\tau^{n}v)$, now suppose $\gamma$ is the union of oriented subpaths $\delta$, $\tau \delta,\ldots,\tau^{n}\delta$, where $\delta$ connects $v$ to $\tau v$. Since $\tau$ preserves signed distance, the second equality follows.\\

The proof is identical for $S_{g,1}$.
\end{proof}





\subsection{Signed Stable Length and Cohomology} 

Choose an element $\tau$ from $\mathcal{T}_g$ or $\mathcal{T}_{g,1}$, and consider it fixed throughout this section. A similar argument to that in the previous section is used to show that the map from $H_{1}(S_{g,1};\mathbb{Z})$ or $H_{1}(S_{g};\mathbb{Z})$ into $\mathbb{Z}$ defined by $\phi$ is a homomorphism.\\

\begin{lem}
\label{pathindependenceagain}
The signed stable length $\phi$ defines an element of $H^{1}(S_{g,1},\mathbb{Z})$ or $H^{1}(S_{g}, \mathbb{Z}/\langle g-1 \rangle)$
\end{lem}
\begin{proof}
This lemma needs to show that $\phi$ is a homomorphism on homology. Let $([\alpha_{1}],\ldots,[\alpha_{g}],[\beta_{1}],\ldots,[\beta_{g}])$ be a set of generators for
$H_{1}(S_{g,1}, \mathbb{Z})$ or $H_{1}(S_{g}, \mathbb{Z})$. A union of oriented curves consisting of $a_1$ curves in the homotopy class $\alpha_1$, $a_2$ curves in the homotopy class $\alpha_2$, etc. will be donoted by $a_{1}\alpha_{1}+a_{2}\alpha_{2}+\ldots$, where a minus sign denotes reversed orientation. Suppose $c$ is a curve representing the homology class $[c]$, where
\begin{equation*}
[c]=\Sigma_{i=1}^{g}a_{i}[\alpha_{i}]+b_{i}[\beta_{i}]
\end{equation*}
For any $n$, a cell complex homotopic to an immersed surface in $S_{g}\times I$ or $S_{g,1}\times I$ with boundary $\tau c-c$ can be constructed by taking a surface $H$ with boundary $\tau (a_{1}\alpha_{1}+\ldots +a_{g}\alpha_{g}+b_{1}\beta_{1}+\ldots +b_{g}\beta_{g})-a_{1}\alpha_{1}-\ldots -a_{g}\alpha_{g}-b_{1}\beta_{1}-\ldots - b_{g}\beta_{g}$, and attaching two surfaces, $F_1$ and $F_2$. The surface $F_1$ has boundary $\tau c-\tau (a_{1}\alpha_{1}+\ldots a_{g}\alpha_{g}+b_{1}\beta_{1}+\ldots b_{g}\beta_{g})$, $F_2$ has boundary $a_{1}\alpha_{1}+\ldots +a_{g}\alpha_{g}+b_{1}\beta_{1}+\ldots + b_{g}\beta_{g}-c$.\\

Since $\partial F_{1}$ is $-\tau(\partial F_{2})$, the contribution of $F_1$ and $F_{2}$ to the signed genus of a surface with boundary $\tau c-c$ cancel out, either absolutely or mod $g-1$, depending on whether the surface is in $S_{g,1}\times I$ or $S_{g}\times I$. This implies that 
\begin{equation*}
\phi(\tau, [c])=\frac{-1}{2}\int p_{H}
\end{equation*}

For $S_{g,1}\times I$ it has been shown that the pre-image function only depends on the boundary of the surface, so $H$ can be chosen to be a union of surfaces $\{A_{i}\}$ and $\{B_{i}\}$, where $A_{i}$ is a union of $a_i$ copies of a surface with boundary $\tau \alpha_{i}-\alpha_{i}$ and $B_{i}$ is a union of $b_i$ copies of a surface with boundary $\tau \beta_{i}-\beta_{i}$. It follows that 
\begin{equation*}
p_{H}=\Sigma_{i=1}^{ g}p_{A_{i}}+p_{B_{i}}
\end{equation*}
For $S_{g,1}$ the lemma follows from equation \ref{hom2}. Similarly for $S_g$.
\end{proof}

\subsection{Signed Stable Length is a Homomorphism on Torelli.} 


The proof that $\phi$ is a homomorphism is very similar to the proof of Lemma \ref{nolimit}.

\begin{lem}
The map $\phi$ is a homomorphism.
\label{homomorphism}
\end{lem}

\begin{proof}
Suppose that $\tau, \tau_{1}$ and $\tau_{2} \in \mathcal{T}_{g,1}$, where $\tau=\tau_{2}\tau_{1}$. By the path independence of signed distance,
\begin{equation*}
d_{s}(c,\tau c)=d_{s}(c, \tau_{2}\tau_{1}c)=d_{s}(c, \tau_{1}c)+d_{s}(\tau_{1}c, \tau_{2}\tau_{1}c).
\end{equation*}
Since the action of the Torelli group preserves signed distances, it follows that 
\begin{equation*}
d_{s}(c, \tau_{1}c)+d_{s}(\tau_{1}c, \tau_{2}\tau_{1}c)=d_{s}(c, \tau_{1}c)+d_{s}(c, \tau_{2}c).
\end{equation*}
The proof is identical for $\tau_1$ and $\tau_2$ in $\mathcal{T}_{g}$.
\end{proof}

\subsection{Signed Stable Length is Zero on the Johnson Kernel}
\label{kernel}
Recall that the Johnson kernel is the kernel of the Johnson homomorphism, and is generated by Dehn twists around separating curves, \cite{Johnson2}. Once we have evaluated signed stable length on bounding pair maps, it would be possible to use the lantern relation to write a Dehn twist around a separating curve as a product of bounding pair maps, and show that the signed distance between any two vertices of $C(S_{g,1}; h)$ or $C(S_{g}; h)$ in the same orbit of the Johnson kernel is zero. However it is also not hard to show this directly from the definition, which is what will be done in this subsection. The basic idea is that a separating curve $s$ has zero algebraic intersection number with any other curve, so Dehn twisting around $s$ essentially adds as many copies of $s$ as it does $-s$.\\

Let $i(a,b)$ denote the geometric intersection number of the curve $a$ with the curve $b$, i.e. the minimum possible number of points of intersection between two curves, one of which is isotopic to $a$ and the other to $b$.  \\

\begin{figure}
\includegraphics[width=\textwidth]{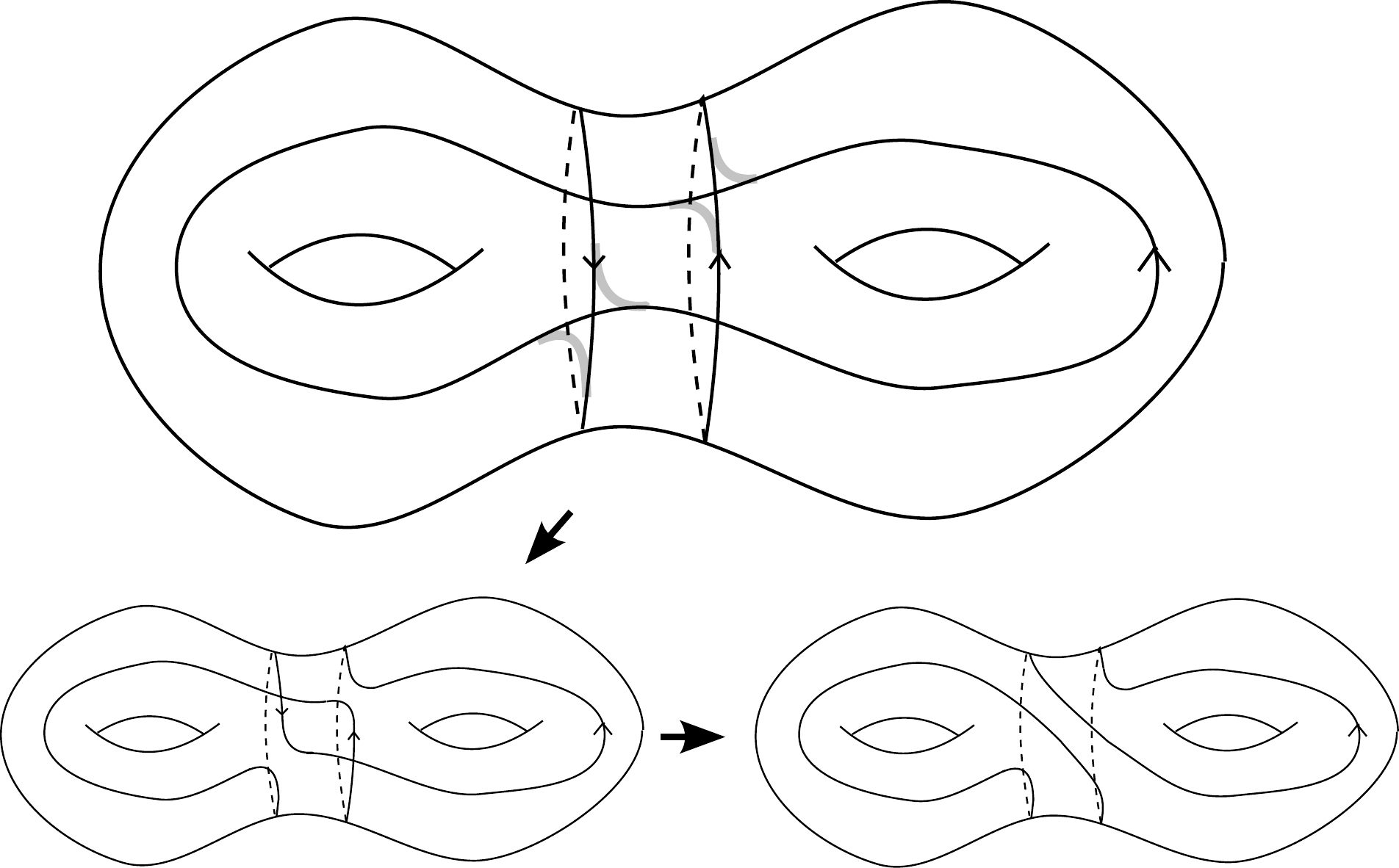}
\caption{The surgeries to resolve the points of intersection are shown in grey in the top figure. The horizontal arrow denotes a homotopy.}
\label{resolve}
\end{figure}

Let $c$ be a curve representing a vertex in $C(S_{g,1}; h)$ or $C(S_{g}; h)$. Up to homotopy, a curve $T_{s}(c)$ can be obtained by taking the union of $c$, $i(c,s)/2$ copies of the curve $s$, and $i(c,s)/2$ copies of $-s$, and performing 1-surgeries to resolve the points of intersection, as illustrated in figure \ref{resolve}. The null homologous union of curves 

\begin{equation*}
T_{s}c-c-\frac{i(c,s)}{2}s+\frac{i(c,s)}{2}s
\end{equation*}
is the boundary of an incompressible surface $F$ in $S_{g,1}\times I$ or $S_{g}\times I$. The support of $p_{F}$ is contained within the union of tubular neighbourhoods of $c$ and tubular neighbourhoods of $i(c,s)$ disjoint curves freely homotopic to $s$. Since the union of tubular neighbourhoods has genus zero, so does $F$. Since
\begin{equation*}
-\frac{i(c,s)}{2}s+\frac{i(c,s)}{2}s
\end{equation*}
is the boundary of a surface with zero signed genus, the same is true for a surface with boundary $T_{s}c-c$.\\

\subsection{Evaluation on the generators.}\label{evaluation} Now that $\phi$ has been shown to be a homomorphism both on the Torelli group and on $H_{1}(S_{g,1},\mathbb{Z})$ and $H_{1}(S_{g},\mathbb{Z}/\langle g-1\rangle)$, all that remains to prove Theorem 1 is to verify it on a choice of generating sets.\\

The Torelli group is generated by bounding pair maps. Let $(a,b)$ be a bounding pair. Choose curves representing a basis $B$ for $H_{1}(S_{g,1},\mathbb{Z})$, so that only one of them, $\alpha$, intersects $a$ and $b$, and $i(a,\alpha)=i(b,\alpha)=1$. Clearly, $d_{s}(\beta, (T_{a}T^{-1}_{b})\beta)=0$ whenever $[\beta]$ is any generator other than $[\alpha]$. The same argument as in Section \ref{kernel} shows that $d_{s}(\alpha, (T_{a}T^{-1}_{b})\alpha)$ is equal to the signed genus of the surface with boundary $a - b$. So $\phi(T_{a}T^{-1}_{b},[\alpha])$ is equal to the signed genus of the subsurface bounded by $a - b$ and $\phi(T_{a}T^{-1}_{b},[\beta])=0$. This is consistent with half the Chillingworth class evaluated on this homology basis, as computed in \cite{Johnson}.\\

For a closed surface the calculation is identical. Let $B$ be a homology basis chosen as in the previous paragraph. It follows from Section 6 of \cite{Johnson}, that this is again equal to half of the Chillingworth class evaluated on the basis $B$.
\end{proof}

This concludes the proof of Theorem \ref{Maintheorem}.\\

\bibliographystyle{plain}
\nocite{*}
\bibliography{Chillbib}

\end{document}